\documentclass{siamart0516}

\usepackage{amsfonts,amssymb,amsopn}
\usepackage{algorithmic}
\usepackage{color,xcolor,graphicx}
\usepackage{multirow,multicol}
\usepackage{enumerate}
\usepackage{epstopdf}
\usepackage{adjustbox}
\usepackage{subfigure}
\usepackage{tikz}
\usepackage{thmtools}
\usepackage{thm-restate}

\usetikzlibrary{shapes.geometric, arrows}

\tikzstyle{point} = [rectangle, rounded corners, minimum width=2cm, minimum height=1cm, draw=black]
\tikzstyle{arrow} = [thick,->,>=stealth]

\newcommand{\R}{\mathbb{R}}

\newcommand{\Z}{\mathbb{Z}}
\newcommand{\Q}{\mathbb{Q}}

\newcommand{\F}{\mathcal{F}}

\newcommand{\IP}{\mathrm{IP}}
\newcommand{\LRp}{\mathrm{LR+}}
\newcommand{\LDp}{\mathrm{LD+}}
\newcommand{\NLP}{\mathrm{NLP}}

\newcommand{\st}{\mathrm{s.t.}}

\newcommand{\conv}{\mathrm{conv}}
\newcommand{\diam}{\mathrm{diam}}
\newcommand{\cone}{\mathrm{cone}}
\newcommand{\icone}{\mathrm{int.cone}}

\newsiamthm{assumption}{Assumption}
\newsiamthm{hypothesis}{Hypothesis}
\newsiamthm{claim}{Claim}
\newsiamthm{remark}{Remark}

\ifpdf
  \DeclareGraphicsExtensions{.eps,.pdf,.png,.jpg}
\else
  \DeclareGraphicsExtensions{.eps}
\fi

\newcommand{\TheTitle}{Exact Augmented Lagrangian Duality for Mixed Integer Quadratic Programming}
\newcommand{\ShortTitle}{Exact Augmented Lagrangian Duality for MIQP}
\newcommand{\TheAuthors}{X. Gu, S. Ahmed, and S.S. Dey}

\headers{\ShortTitle}{\TheAuthors}

\title{{\TheTitle}
}

\author{
  Xiaoyi Gu\thanks{H. Milton Stewart School of Industrial and Systems Engineering, Georgia Institute of Technology, Atlanta, GA, USA (\email{xiaoyigu@gatech.edu}).}
  \and
  Shabbir Ahmed\thanks{H. Milton Stewart School of Industrial and Systems Engineering, Georgia Institute of Technology, Atlanta, GA, USA (\email{sahmed@isye.gatech.edu}).}
  \and
  Santanu S. Dey\thanks{H. Milton Stewart School of Industrial and Systems Engineering, Georgia Institute of Technology, Atlanta, GA, USA (\email{santanu.dey@isye.gatech.edu}).}
}

\ifpdf
\hypersetup{
  pdftitle={\TheTitle},
  pdfauthor={\TheAuthors}
}
\fi

\begin{document}

\maketitle

\begin{abstract}
Mixed integer quadratic programming (MIQP) is the problem of minimizing a convex quadratic function over mixed integer points in a rational polyhedron. 
This paper focuses on the augmented Lagrangian dual (ALD) for MIQP. ALD augments the usual Lagrangian dual with a weighted nonlinear penalty on the dualized constraints. We first prove that ALD will reach a zero duality gap asymptotically as the weight on the penalty goes to infinity under some mild conditions on the penalty function. We next show that a finite penalty weight is enough for a zero gap when we use any norm as the penalty function. Finally, we prove a polynomially bound on the weight on the penalty term to obtain a zero gap. 
 \end{abstract}
%


\section{Introduction}
\label{sec:intro}
We consider the following rational (mixed) integer quadratic programming (MIQP) problem with decision variable $x\in\R^n$:
\begin{equation}\label{eq:mainproblem}
z^{\IP}:=\inf\{c^\top x+\frac{1}{2}x^\top Q x:Ax=b,x\in X\},  
\end{equation}
where the parameters are: a rational symmetric positive semi-definite matrix $Q\in\Q^{n\times n}$, a rational matrix $A\in\Q^{m\times n}$, rational vectors $c\in \Q^n$ and $b\in \Q^m$, a mixed integer linear set $X$ such that
\[X=\{(x_1,x_2)\in\R^{n_1}\times \Z^{n_2}:Ex\leq f\},\]
where $E\in\Q^{m_2\times n}$ is a rational matrix and $f\in\Q^{m_2}$ is a rational vector with $n_1+n_2=n$. 
%
We consider dualizing the constraints $Ax= b$.

While for continuous quadratic programming (QP), it is well known that even the classical Lagrangian dual (LD) will reach a zero duality gap and strong duality holds \cite{bertsekas2003convex}, it is not true for MIQP, as the integer variables introduce non-convexity. In fact, LD may have a non-zero duality gap for the problem. 
Therefore, to close the gap, the idea of penalizing violation of the dualized constraints with a nonlinear penalty gives rise to the well known augmented Lagrangian dual (ALD), which is
\[z_\rho^{\LDp}:=\sup_\lambda\inf_{x\in X}\{c^\top x+\frac{1}{2}x^\top Q x+\lambda^\top(b-Ax)+\rho\psi(b-Ax)\},\]
where $\rho>0$ is the penalty weight, and $\psi(\cdot)$ is the penalty function which usually satisfies $\psi(0)=0$ and $\psi(u)>0$ if $u\neq 0$ \cite{rockafellar2009variational}.

Numerous papers have discussed ALD. 
The paper \cite{rockafellar1974augmented} uses convex quadratic penalty functions for nonconvex programming, \cite{feizollahi2017exact} discusses the asymptotic zero duality gap and exact penalty representation for mixed integer \emph{linear} programming (MILP), \cite{burachik2017existence} discusses the optimality conditions for semi-infinite programming, and \cite{burachik2010duality} discusses exact penalization for general augmented Lagrangian.

It should be noted that an exact penalty representation usually requires a much restricted penalty function, like norm functions, see for example \cite{rockafellar2009variational}. 
Norm function is used in \cite{burke1991exact} for exact penalization. The work \cite{huang2003unified} discusses exact penalty representation using level-bounded augmented functions and \cite{rubinov2002zero} considers the penalty function which is almost peak at zero. More recent works like \cite{feizollahi2017exact,burachik2017existence} apply sharp Lagrangian to different types of problems.

On the other hand, the size (for example, in binary coding) of the penalty weight is rarely discussed. While there are discussions for the size and computational complexity of MILP \cite{von1978bound,borosh1976bounds}, QP \cite{vavasis1990quadratic} and MIQP \cite{del2017mixed}, we might be able to utilize their ideas to show the small size of the penalty weight.

In this paper, we significantly generalize the results of \cite{feizollahi2017exact}. In particular, we
\begin{enumerate}
\item Prove that the duality gap of ALD will asymptotically reach zero under mild conditions as the penalty weight goes to infinity;
\item Prove that the duality gap will reach zero given that the penalty function is any norm, and the penalty weight is sufficiently large but still finite;
\item Prove that the size of the penalty weight which attains zero duality gap is polynomially bounded with respect to the problem data. 
\end{enumerate}

The paper is organized as follows. In \Cref{sec:mainresults} we provide definitions and formal statement of main results of the paper. In \Cref{sec:preliminary} we present several key lemmas useful across the paper. In \Cref{sec:zerogap} we exhibit properties of ALD as the penalty weight goes to infinity, and show the (asymptotic) zero duality gap for a large class of penalty functions. In \Cref{sec:exactpenalty} we show a finite penalty weight whose size is polynomially bounded with respect to the input parameters, for which a zero duality gap is attained.
\section{Main Results}
\label{sec:mainresults}
In this section, we introduce some definitions and briefly present our main results.

\begin{assumption}\label{ass:problem}
The MIQP \cref{eq:mainproblem} is feasible and the optimal value is bounded.
\end{assumption}

\begin{definition} The augmented Lagrangian relaxation is defined as
\[z_\rho^{\LRp}(\lambda):=\inf_{x\in X}\{c^\top x+\frac{1}{2}x^\top Q x+\lambda^\top(b-Ax)+\rho\psi(b-Ax)\},\]
where $\psi$ is a penalty function. The augmented Lagrangian dual is defined as
\[z_\rho^{\LDp}:=\sup_\lambda z_\rho^{\LRp}(\lambda)=\sup_\lambda\inf_{x\in X}\{c^\top x+\frac{1}{2}x^\top Q x+\lambda^\top(b-Ax)+\rho\psi(b-Ax)\}.\]
\end{definition}

\begin{definition} The continuous relaxation of \cref{eq:mainproblem} is denoted as $z^{\NLP}$
\[z^\NLP:=\inf\{c^\top x+\frac{1}{2}x^\top Q x: Ax=b,Ex\leq f,x\in\R^{n_1+n_2}\}.\]
\end{definition}

\begin{remark}
We use $\bar{\lambda}$ to denote the optimal dual variables (of $z^\NLP$) for the constraints $Ax=b$ and $\bar{\lambda}_E$ to denote the optimal dual variables for $Ex\leq f$. The existence of $\bar{\lambda}$ and $\bar{\lambda}_E$ is guaranteed by the boundedness of the continuous relaxation, which is given by \cref{lm:NLPQPIP}.
\end{remark}

\begin{remark}
For any $\rho$, $\lambda$, we have $z_\rho^{\LRp}(\lambda)\leq z_\rho^{\LDp}\leq z^\IP$. Moreover, we have $z^\NLP=\inf\{c^\top x+\frac{1}{2}x^\top Q x+\bar{\lambda}^\top(b-Ax):Ex\leq f,x\in\R^{n_1+n_2}\}\leq z_\rho^{\LRp}(\bar{\lambda}) \leq z_\rho^{\LDp}\leq z^\IP$.
\end{remark}

\begin{definition}
For a finite set of vectors $T=\{t_1,t_2,...,t_k\}$, $\conv(T)$, $\cone(T)$ and $\icone(T)$ are the convex hull, conical hull and integral conical hull of $T$, respectively. Here, $\icone(T):=\{\sum_{i=1}^k\mu_it_i:\mu_i\in\Z_+\}$.
\end{definition}

\begin{definition}
For any subset $T$ of a metric space, its diameter $\rm{\diam}$ is defined as ${\rm{\diam}}(T)=\sup _{a,b\in T}\|a-b\|$, where $\|\cdot\|$ is the metric associated with the space. 
\end{definition}

\begin{definition} [\cite{del2017mixed}] Given an object $\mathcal{O}$ and another object $f(\mathcal{O})$ which is a function of it, we say that $f(\mathcal{O})$ has $\mathcal{O}$-small complexity, if the size (in standard binary encoding) of $f(\mathcal{O})$ is bounded above by a a polynomial function of the size of $\mathcal{O}$.
\end{definition}

\begin{definition} We use $\F$ to denote all input parameters of \cref{eq:mainproblem} including $E$, $f$, $c$, $Q$, $A$ and $b$. In addition, any object $q$ which is a function of $\F$ is said to have small complexity, if $q$ has $\F$-small complexity.
\end{definition}

Below we present the main theorems of the paper.

\begin{restatable}[Asymptotic Zero Duality Gap]{theorem}{AZDP}
\label{thm:main:zdg}
\label{thm:zdg}
Assume $\psi$ is proper, nonnegative, lower-semicontinuous and level-bounded, that is: $\psi(0)=0$; $\psi(u)>0$ for all $u\neq 0$; $\lim_{\delta\downarrow 0}{\rm{diam}}\{u:\psi(u)\leq \delta\}=0$; ${\rm{diam}}\{u:\psi(u)\leq \delta\}<\infty$ for all $\delta>0$. We have $\sup_{\rho>0}z^{\LDp}_\rho=z^{\IP}$.
\end{restatable}
\begin{restatable}[Exact Penalty Representation]{theorem}{EPR}
\label{thm:main:flexeald}
\label{thm:flexeald}
Suppose $\psi(\cdot)$ is any norm. 
\begin{enumerate}[(a)]
\item There exists a $\rho^*$ of small complexity, such that $z^{\LDp}_{\rho^*}=z^{\IP}$.
\item Moreover, for all $\lambda$, there exists a finite $\rho^*(\lambda)$ of $\F,\lambda$-small complexity, such that $z^\LRp_{\rho^*}(\lambda)=z^{\IP}$.
\end{enumerate}
\end{restatable}

We provide a flowchart that depicts how the preliminary results proved in \Cref{sec:preliminary} are put together to prove \cref{thm:main:zdg}. 

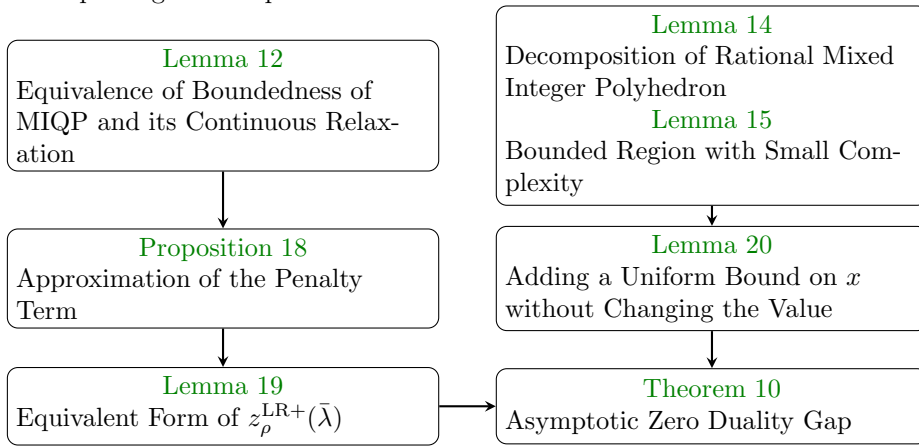
\begin{tikzpicture}[node distance=2cm, text width=5.5cm]
\node (l12) [point] {\centering\cref{lm:NLPQPIP} \\ Equivalence of Boundedness of MIQP and its Continuous Relaxation};
\node (t17) [point, below of=l12, yshift = -0.3cm] {\centering\cref{prop:penaltylimit} \\ Approximation of the Penalty Term};
\node (t18) [point, below of=t17, yshift = 0.3cm] {\centering\cref{lm:penaltyrelax} \\ Equivalent Form of $z^\LRp_\rho(\bar{\lambda})$};
\node (l1314) [point, right of=l12, xshift = 4.5cm] {{\centering\cref{lm:Qdecomp}\\} Decomposition of Rational Mixed Integer Polyhedron \\ \centering\cref{lm:boundedref} \\ Bounded Region with Small Complexity};
\node (t20) [point, below of=l1314, yshift = -0.3cm] {\centering\cref{lm:zdg:boundary} \\ Adding a Uniform Bound on $x$ without Changing the Value};
\node (t19) [point, below of=t20, yshift = 0.3cm] {\centering\cref{thm:zdg}\\ Asymptotic Zero Duality Gap};

\draw [arrow](l12)--(t17);
\draw [arrow](t17)--(t18);
\draw [arrow](t18)--(t19);
\draw [arrow](l1314)--(t20);
\draw [arrow](t20)--(t19);
\end{tikzpicture}

Another flowchart is provided respect to the proof of \cref{thm:main:flexeald}.

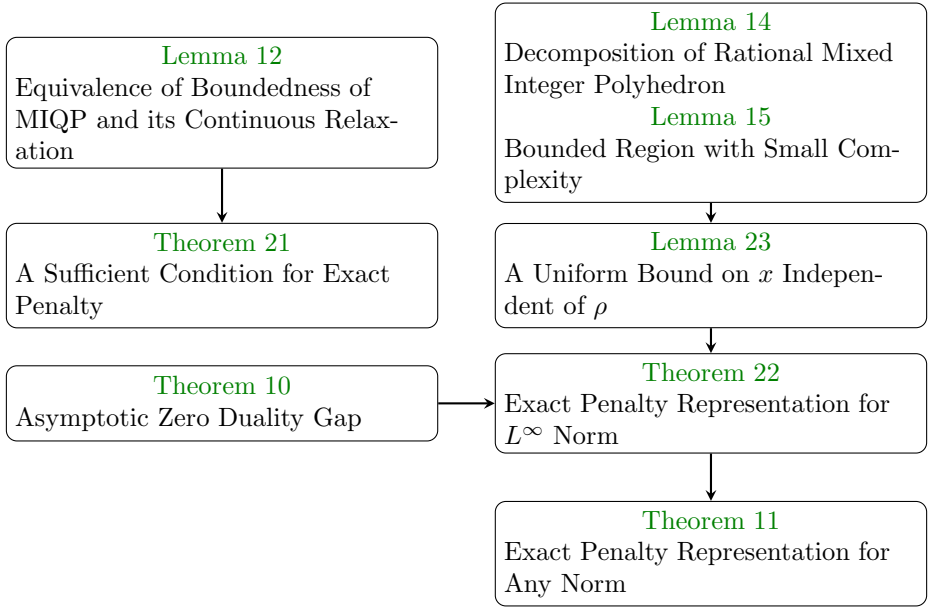
\begin{tikzpicture}[node distance=2cm, text width=5.5cm]
\node (l12) [point] {\centering\cref{lm:NLPQPIP}\\ Equivalence of Boundedness of MIQP and its Continuous Relaxation};
\node (t21) [point, below of=l12, yshift = -0.3cm] {\centering\cref{thm:fixeald_ge}\\ A Sufficient Condition for Exact Penalty};
\node (t19) [point, below of=t21, yshift = 0.3cm] {\centering\cref{thm:zdg}\\ Asymptotic Zero Duality Gap};
\node (l1314) [point, right of=l12, xshift = 4.5cm] {{\centering\cref{lm:Qdecomp}\\} Decomposition of Rational Mixed Integer Polyhedron \\ \centering\cref{lm:boundedref} \\ Bounded Region with Small Complexity};
\node (t23) [point, below of=l1314, yshift = -0.3cm] {\centering\cref{lm:eald:boundary} \\ A Uniform Bound on $x$ Independent of $\rho$};
\node (t22) [point, below of=t23, yshift = 0.3cm] {\centering\cref{thm:fixeald}\\ Exact Penalty Representation for $L^\infty$ Norm};
\node (t24) [point, below of=t22] {\centering\cref{thm:flexeald}\\ Exact Penalty Representation for Any Norm};

\draw [arrow](l12)--(t21);
\draw [arrow](t19)--(t22);
\draw [arrow](l1314)--(t23);
\draw [arrow](t23)--(t22);
\draw [arrow](t22)--(t24);
\end{tikzpicture}
\section{Preliminary Results}
\label{sec:preliminary}
Several useful lemmas are presented in this section.
\begin{lemma}[Equivalence of Boundedness of MIQP and its Continuous Relaxation]
\label{lm:NLPQPIP}
Suppose the MIQP is feasible (i.e. $z^{\IP}<+\infty$). Then the following three conditions are equivalent:
\begin{enumerate}
\item $z^{\NLP}$ is bounded.
\item $\inf\{c^\top x|Ax=0, Ex\leq 0, Qx=0\}$ is bounded.
\item $z^{\IP}$ is bounded.
\end{enumerate}
\end{lemma}

\begin{proof}
$1\Rightarrow3$ is obvious.

We first show $3\Rightarrow2$, or equivalently $\neg 2\Rightarrow \neg 3$. Note that the problem in 2 is always feasible. Assuming $\neg 2$, the problem $\{c^\top x\leq -1, Ax=0,Ex\leq 0,Qx=0\}$ is now feasible and there exists a rational solution since the problem is rational. Denote such a rational solution as $r$ and without loss of generality, we assume that $r$ is integral since we can scale $r$ with a positive coefficient.

Now select any feasible solution for 3, as $x$. Then we know that $x+tr$ is still feasible for 3 for any $t\in\Z_+$. In addition, $c^\top(x+tr)+\frac{1}{2}(x+tr)^\top Q (x+tr)=c^\top x+\frac{1}{2}x^\top Qx+tc^\top r\to-\infty$ as $t\to+\infty$. Therefore, we have 3 is unbounded, i.e. $3\Rightarrow 2$.

Next we show that $2\Rightarrow1$. Suppose that 2 holds. From Farkas Lemma, we know that $\exists \lambda_E\leq 0, \lambda_A, \lambda_Q$, such that
$\lambda_E^\top E+\lambda_A^\top A+\lambda_Q^\top Q=c^\top$. Now considering the NLP
\[
\begin{aligned}
z^{\NLP}=\inf\ &c^\top x+\frac{1}{2}x^\top Q x\\
\st\ &Ax=b\\
&Ex\leq f\\
=\inf\ &(\lambda_E^\top E+\lambda_A^\top A+\lambda_Q^\top Q)x+\frac{1}{2}x^\top Qx\\
\st\ &Ax=b\\
&Ex\leq f\\
\geq\inf\ &\lambda_E^\top f+\lambda_A^\top b + \lambda_Q^\top Qx+\frac{1}{2}x^\top Qx\\
\st\ &Ax=b\\
&Ex\leq f.
\end{aligned}
\]
To show that $\lambda_Q^\top Q x +\frac{1}{2} x^\top Q x$ is bounded, we first write down the orthogonal decomposition of $Q$ as $R^\top\Lambda R$ where $R$ is orthogonal and $\Lambda$ is diagonal. Therefore, $\lambda_Q^\top Q x +\frac{1}{2} x^\top Q x=\sum_i \Lambda_{ii}(\lambda_Q^\top r_i+\frac{1}{2}x^\top r_i)x^\top r_i\geq \sum_i-\Lambda_{ii} (\lambda_Q^\top r_i)^2/2=-\lambda_Q^\top Q\lambda_Q/2$, where $r_i$ is the $i$-th row of the orthogonal matrix $R$. Note that the bound is attainable if we take $x=-\lambda_Q$. Therefore, we arrive at 1, i.e. $z^\NLP$ is bounded.
\end{proof}

\begin{remark} We note here that we are able to prove that the boundedness of the nonlinear integer problem implies boundedness of its continuous relaxation, using the fact that the data is rational. This is very similar to the Fundamental theorem of Integer Programming \cite{meyer1974existence}. Note that other similar results may be proven under different assumptions such as existence of integer point in the interior of  continuous relaxation, see \cite{dey2013some,moran2012strong}. Also see \cite{moran2018subadditive}.
\end{remark}

\begin{lemma}[Decomposition of Rational Mixed Integer Polyhedron]
\label{lm:Qdecomp}
Given a rational positive semidefinite matrix $Q$, any rational mixed integer polyhedron $P\cap(\R^{n_1}\times \Z^{n_2})=\{x:Cx\leq d\}\cap(\R^{n_1}\times \Z^{n_2})$ can be decomposed (with respect to $Q$) as $\cup_i (P_i\cap(\R^{n_1}\times \Z^{n_2})+\icone(R_i))$ satisfying the following properties:
\begin{enumerate}[(a)]
\item Each $P_i$ is a rational polytope.
\item Each $\cone(R_i)$ is a rational, simple and pointed cone.
\item For every $\cone(R_i)$, if a face $C'$ satisfies that $\exists x\in C'\backslash\{0\}$, $x^\top Qx=0$, then there exists an extreme ray $v$ of $C'$ with $v^\top Qv=0$.
\item Each polytope $P_i$ and each vector in $R_i$ has $P,Q$-small complexity.
\end{enumerate}
\end{lemma}

\begin{proof}
This lemma is a direct outcome of \cite[Proposition 1, Proposition 2, Lemma 2]{del2017mixed}.

First, if $P$ is not pointed, we can decompose $P$ into at most $2^{n_1+n_2}$ pointed rational mixed integer polyhedron by separating $x_k\leq 0$ and $x_k\geq 0$ for all $k$. Therefore, we simply assume $P$ is pointed henceforth.

Next, using \cite[Proposition 1]{del2017mixed}, we can decompose $P$ as $P=\cup_{i,K^1\in\mathcal{K}^1}(P^1_i+\cone(R^1_{K^1}))$, while conditions (a), (b), (d) are met.

Later, using \cite[Lemma 2]{del2017mixed}, we are able to decompose $\cone(R_k)$ into a union of rational, simple and pointed cones, which satisfies condition (c) and maintains (a), (b), (d). Therefore, $P=\cup_{i,K^2\in\mathcal{K}^2}(P^1_i+\cone(R^2_{K^2}))$.

Finally, we use \cite[Proposition 2]{del2017mixed} and decompose $(P^1_i+\cone(R^2_{K^2})) \cap(\R^{n_1}\times \Z^{n_2})$ into a mixed integer rational polytope plus an integer cone, which completes the proof.
\end{proof}

\begin{lemma}[Bounded Region with Small Complexity \cite{del2017mixed}]
\label{lm:boundedref}
Let $P \subseteq R^n$ be a polytope and $R\subseteq R^n$ be a finite set of vectors. Given a rational positive semidefinite matrix $Q$ of small complexity, suppose $P+\cone(R)$ satisfies the following properties:
\begin{enumerate}[(a)]
\item $P$ is a rational polytope.
\item $\cone(R)$ is a rational, simple and pointed cone.
\item $\forall x\in\cone(R)\backslash\{0\}$, $x^\top Qx>0$.
\item $P$ and each vector in $R$ has small complexity.
\end{enumerate}
Then, for any $\eta\in\R^n,\mu\in\R$ of small complexity, there exists $M$ of small complexity such that $\{x\in P+\cone(R):
\frac{1}{2}x^\top Qx+\eta^\top x\leq\mu\} \subset \{x:\|x\|\leq M\}$. In addition, such $M$ exists for any norm.
\end{lemma}

\begin{remark}
For any rational mixed integer polyhedron, \cref{lm:Qdecomp} provides a decomposition with respect to $Q$, while maintaining a small complexity. In addition, for any part of the decomposition, if no extreme ray $v$ has $v^\top Q v=0$ then no ray has $x^\top Q x=0$ (i.e. $x^\top Q x>0$ for any ray).

\cref{lm:boundedref} shows that under the conditions that $x^\top Q x>0$ for any ray, the optimal solution of the optimization problem $\{\min\frac{1}{2}x^\top Q x+\eta^\top x:x\in P+\cone(R)\}$ will have small complexity. In the lemma, this property is presented in the form of a feasibility problem.

The two lemmas will be needed for proving bounds in our proofs of theorems.
\end{remark}
\section{Asymptotic Zero Duality Gap}
\label{sec:zerogap}
In this section, we show that under mild conditions on the penalty function the ALD duality gap vanishes as the penalty weight $\rho$ goes to infinity. 
%
\begin{assumption}[Conditions for Asymptotic Zero Duality Gap]
\label{assum:penalty}
We assume $\psi$ is proper, nonnegative, lower-semicontinuous and level-bounded, that is: $\psi(0)=0$; $\psi(u)>0$ for all $u\neq 0$; $\lim_{\delta\downarrow 0}{\rm{diam}}\{u:\psi(u)\leq \delta\}=0$; ${\rm{diam}}\{u:\psi(u)\leq \delta\}<\infty$ for all $\delta>0$.
\end{assumption}
\begin{proposition}[Approximation of the Penalty Term]
\label{prop:penaltylimit}
For given $\rho>0$ and $\epsilon>0$, define $w_{\rho,\epsilon}^*$ as
\begin{equation}
\label{eq:prop:penaltylimit}
\begin{aligned}
w_{\rho,\epsilon}^* :=\ \inf_{x,w}\ &w\\
\st\ &x\in X,\\
\ &\psi(b-Ax)\leq w,\\
\ &c^\top x+\frac{1}{2}x^\top Qx+\bar{\lambda}^\top(b-Ax)+\rho w-z_\rho^{\LRp}(\bar{\lambda})\leq\epsilon.
\end{aligned}
\end{equation}
Then, the limit $w^*_\rho:=\lim_{\epsilon\downarrow 0}w_{\rho,\epsilon}^*$ exists and $\lim_{\rho\to+\infty}w^*_\rho=0$.
\end{proposition}
\begin{proof}First we need show that the problem \cref{eq:prop:penaltylimit} is well-defined, i.e is feasible and bounded. As a first step we show that $z_\rho^{\LRp}(\bar{\lambda})$ is finite. Observe that:
\[
\begin{aligned}
z_\rho^{\LRp}(\bar \lambda )\geq &\inf\{c^\top x+\frac{1}{2}x^\top Qx+\bar{\lambda}^\top (b-Ax):x\in X\}\\
\geq &\inf\{c^\top x+\frac{1}{2}x^\top Qx+\bar{\lambda}^\top (b-Ax):Ex\leq f\}\\
\geq &\inf\{c^\top x+\frac{1}{2}x^\top Qx+\bar{\lambda}^\top (b-Ax) +\bar{\lambda}^\top_E (f-Ex)\}\\
=&z^\NLP,
\end{aligned}
\]
and the boundedness of $z^\NLP$ is given by \cref{lm:NLPQPIP}.

From the feasibility of the original problem we know that there exists an $x$ feasible for $z_\rho^{\LRp}(\bar{\lambda})$. Therefore, we are able to find $x\in X$ such that $c^\top x+\frac{1}{2}x^\top Qx+\bar{\lambda}^\top (b-Ax)+\psi(b-Ax)\leq \epsilon+ z_\rho^{\LRp}(\bar{\lambda})$, which means $(x,w=\psi(b-Ax))$ is feasible for \cref{eq:prop:penaltylimit}.
We also have $w^*_{\rho,\epsilon}\geq 0$ from the non-negativity of $\psi$. Thus, \cref{eq:prop:penaltylimit} is feasible and bounded.

In addition, we have $z_\rho^{\LRp}(\bar{\lambda})\leq z^{\IP}$ and for $x$ satisfying $Ex \leq f$ we have that $c^\top x+\frac{1}{2}x^\top Qx+\bar{\lambda}^\top(b-Ax)\geq \inf\{c^\top x+\frac{1}{2}x^\top Qx+\bar{\lambda}^\top (b-Ax) +\bar{\lambda}^\top_E (f-Ex)\}=z^\NLP$. Therefore,
\begin{align*}
w_{\rho,\epsilon}^*\leq& \frac{1}{\rho} \{z_\rho^{\LRp}(\bar{\lambda})+\epsilon -[c^\top x+\frac{1}{2}x^\top Qx+\bar{\lambda}^\top(b-Ax):x\in X]\}\\
\leq &\frac{1}{\rho}(z^{\IP}+\epsilon-z^{\NLP}).
\end{align*}

By taking $\epsilon\downarrow 0$ we have
\begin{equation}\label{eq:wrhostar}
0\leq w_\rho^*=\lim_{\epsilon\downarrow 0} w_{\rho,\epsilon}^*\leq\lim_{\epsilon\downarrow 0}\frac{1}{\rho}(z^{\IP}+\epsilon-z^{\NLP})=\frac{1}{\rho}(z^{\IP}-z^{\NLP}).\end{equation}
In addition, as $\epsilon\downarrow 0$ the feasible region of \cref{eq:prop:penaltylimit} becomes smaller, which indicates that $w_{\rho,\epsilon}^*$ is non-decreasing. Therefore $w_\rho^*=\lim_{\epsilon\downarrow 0} w_{\rho,\epsilon}^*$ exists. 

By taking $\rho\to+\infty$ we therefore obtain $\lim_{\rho\to+\infty} w^*_\rho=0$.
\end{proof}

\begin{lemma}[Equivalent Form of $z^\LRp_\rho(\bar{\lambda})$]\label{lm:penaltyrelax}
Consider $w^*_\rho$ as in \cref{prop:penaltylimit} and define $\tilde{z}_\rho^{\LRp}(\bar{\lambda})$ as
\begin{equation}\label{eq:penaltyrelax:tildez}
\begin{aligned}
\tilde{z}_\rho^{\LRp}(\bar{\lambda}):=\ \inf_{x,w}\ &c^\top x +\frac{1}{2}x^\top Q x+\bar{\lambda}^\top (b-Ax)+\rho w\\
\st\ &x\in X,\\
&\psi(b-Ax)\leq w,\\
&(1-\delta)w_\rho^*\leq w\leq(1+\delta)w_\rho^*.
\end{aligned}
\end{equation}
Then, for any $\delta\in(0,1)$,
\begin{equation}
\label{eq:LRptb}
\begin{aligned}
z_\rho^{\LRp}(\bar{\lambda})=&\tilde{z}_\rho^{\LRp}(\bar{\lambda})\\
\geq&\inf_x\ c^\top x+\frac{1}{2} x^\top Q x+\bar{\lambda}^\top(b-Ax)+\rho(1-\delta)w_\rho^*\\
&\st\ x\in X,\\
&\ \ \ \ \ \psi(b-Ax)\leq (1+\delta)w_\rho^*,\\
\geq&\inf_x\ c^\top x+\frac{1}{2} x^\top Q x+\bar{\lambda}^\top(b-Ax)\\
&\st\ x\in X,\\
&\ \ \ \ \ \psi(b-Ax)\leq (1+\delta)w_\rho^*.
\end{aligned}
\end{equation}
\end{lemma}
\begin{proof}
Note that the definition of $\tilde z_{\rho}^{\LRp} (\bar \lambda)$ is the same as that of $z_{\rho}^{\LRp} (\bar \lambda)$ except for the additional constraint $(1 - \delta)w^*_{\rho} \leq w \leq (1 + \delta)w^*_{\rho}$. Suppose $\alpha_\rho:=\tilde{z}_\rho^{\LRp}(\bar{\lambda})-z_\rho^{\LRp}(\bar{\lambda})>0$ by contradiction. Then, for all $(x,w)$ feasible to \cref{eq:penaltyrelax:tildez} we have
\[c^\top x +\frac{1}{2}x^\top Q x+\bar{\lambda}^\top (b-Ax)+\rho w\geq \tilde{z}_\rho^{\LRp}(\bar{\lambda})=z_\rho^{\LRp}(\bar{\lambda})+\alpha_\rho,\]
which implies $(x,w)$ is infeasible for \cref{eq:prop:penaltylimit} if $\epsilon<\alpha_\rho$. Hence, $w^*_{\rho,\epsilon}\notin ((1-\delta)w_\rho^*,(1+\delta)w_\rho^*$), a contradiction. Therefore $\tilde{z}_\rho^{\LRp}(\bar{\lambda})=z_\rho^{\LRp}(\bar{\lambda})$ and the inequalities \cref{eq:LRptb} are straightforward to verify.
\end{proof}

We are now ready to present the asymptotic zero duality gap.
\AZDP*
\begin{proof}
$z^{\LDp}_\rho$ does not decrease as $\rho$ increases. Therefore, it is then sufficient to show that $\sup_{\rho\geq 1}z^{\LDp}_\rho=z^{\IP}$ under the assumption.

Let $\delta\in(0,1)$, and we have
\begin{subequations}
\begin{align}
  z_\rho^{\LDp}&\geq z_\rho^{\LRp}(\bar{\lambda}) \nonumber\\
  &\geq\inf_{x}\{c^\top x+\frac{1}{2}x^\top Qx+\bar{\lambda}^\top(b-Ax):x\in X,\psi(b-Ax)\leq(1+\delta)w_\rho^*\} \label{eq:zdg:1a} \\
  &\geq\inf_{x}\{c^\top x+\frac{1}{2}x^\top Qx+\bar{\lambda}^\top(b-Ax):x\in X,\psi(b-Ax)\leq\frac{2}{\rho}(z^\IP-z^\NLP)\} \label{eq:zdg:1b} \\
  &\geq\inf_{x}\{c^\top x+\frac{1}{2}x^\top Qx+\bar{\lambda}^\top(b-Ax):x\in X,\|b-Ax\|_\infty\leq\kappa_\rho\} (\geq z^\NLP) \label{eq:zdg:1c} 
\end{align}
\end{subequations}
where $\kappa_\rho:=\diam\{u:\psi(u)\leq\frac{2}{\rho}(z^\IP-z^\NLP)\}$ which is obviously non-increasing with respect to $\rho$. \cref{eq:zdg:1a} is guaranteed by \cref{lm:penaltyrelax}. \cref{eq:zdg:1b} is valid from \cref{eq:wrhostar} and \cref{eq:zdg:1c} comes from the level-boundedness of \cref{assum:penalty}. 

We will need the following lemma that provides a uniform bound $M$ on \cref{eq:zdg:1c} for $x$ independent of $\rho$.

\begin{lemma}[Adding a Uniform Bound on $x$ without Changing the Value]\label{lm:zdg:boundary}
Under the assumption that $\rho\geq 1$, $\exists M>0$ independent of $\rho$, such that
\[
\begin{aligned}
&\inf_{x}\{c^\top x+\frac{1}{2}x^\top Qx+\bar{\lambda}^\top(b-Ax):x\in X,\|b-Ax\|_\infty\leq\kappa_\rho\}\\
=&\min_{x}\{c^\top x+\frac{1}{2}x^\top Qx+\bar{\lambda}^\top(b-Ax):x\in X,\|b-Ax\|_\infty\leq\kappa_\rho,\|x\|_\infty\leq M\}.
\end{aligned}
\]
\end{lemma}

A proof of \cref{lm:zdg:boundary} is provided later. From \cref{lm:zdg:boundary} we have $z_\rho^\LDp\geq\min_{x}\{c^\top x+\frac{1}{2}x^\top Qx+\bar{\lambda}^\top(b-Ax):x\in X,\|b-Ax\|_\infty\leq\kappa_\rho,\|x\|_\infty\leq M\}$. By taking $\rho\to+\infty$, we get
\begin{subequations}
\begin{align}
&\lim_{\rho\to+\infty}z_\rho^{\LDp} \nonumber\\
\geq&\lim_{\rho\to+\infty}\min_{x}\{c^\top x+\frac{1}{2}x^\top Qx+\bar{\lambda}^\top(b-Ax):x\in X,\|b-Ax\|_\infty\leq\kappa_\rho,\|x\|_\infty\leq M\} \nonumber\\
=&\lim_{\rho\to+\infty}\min_{\|x_2\|_\infty\leq M}\min_{\|x_1\|_\infty\leq M}\{c^\top x+\frac{1}{2}x^\top Qx+\bar{\lambda}^\top(b-Ax):x\in X,\|b-Ax\|_\infty\leq\kappa_\rho\} \nonumber\\
=&\min_{\|x_2\|_\infty\leq M}\lim_{\rho\to+\infty}\min_{\|x_1\|_\infty\leq M}\{c^\top x+\frac{1}{2}x^\top Qx+\bar{\lambda}^\top(b-Ax):x\in X,\|b-Ax\|_\infty\leq\kappa_\rho\} \label{eq:zdg:enumeration}\\
\geq&\min_{\|x_2\|_\infty\leq M}\min_{\|x_1\|_\infty\leq M}\{c^\top x+\frac{1}{2}x^\top Qx+\bar{\lambda}^\top(b-Ax):x\in X,\|b-Ax\|_\infty\leq\lim_{\rho\to+\infty}\kappa_\rho\} \label{eq:zdg:lsc} \\
\geq&\min_{\|x\|_\infty\leq M}\{c^\top x+\frac{1}{2}x^\top Qx+\bar{\lambda}^\top(b-Ax):x\in X,\|b-Ax\|_\infty=0\} \label{eq:zdg:level} \\
=&\min_{x}\{c^\top x+\frac{1}{2}x^\top Qx:x\in X,Ax=b,\|x\|_\infty\leq M\}\geq z^{\IP}, \nonumber
\end{align}
\end{subequations}
where \cref{eq:zdg:enumeration} follows from the finiteness of $x_2$ under $\|x_2\|_\infty\leq M$, \cref{eq:zdg:lsc} follows from the lower semi-continuity of the continuous quadratic programming \cite[Proposition 6.5.2]{bertsekas2003convex}
and \cref{eq:zdg:level} holds by \cref{assum:penalty}. Note that for any $\rho,\lambda$, $z_\rho^{\LRp}(\lambda)\leq z_\rho^{\LDp}\leq z^{\IP}$, and thus $\lim_{\rho\to+\infty}z_\rho^{\LDp}=z^\IP$.
\end{proof}

Note that by proving the theorem we also show that $\lim_{\rho\to+\infty}z_\rho^{\LRp}(\bar{\lambda})=z^\IP$ from the non-decreasing of $z_\rho^{\LRp}(\bar{\lambda})$ with respect to $\rho$.

We now complete the proof by proving \cref{lm:zdg:boundary}
\begin{proof}[Proof of \cref{lm:zdg:boundary}]
Note that $\|b-Ax\|_\infty\leq\kappa_\rho$ can be written as linear constraints. Hence, apply \cref{lm:Qdecomp} to the feasible region for $\rho=1$ of \cref{eq:zdg:1b} and we get a decomposition $\cup_i (P_i\cap(\R^{n_1}\times \Z^{n_2})+\icone(R_i))$ with the properties listed in the lemma. Note that for all $r\in R_i$ we have $Ar=0$ from the constraints $\|b-Ax\|_\infty\leq\kappa_1$. Therefore, the feasible region for any $\rho\geq 1$ can be written as 
\[
\begin{aligned}
&(\cup_i (P_i\cap(\R^{n_1}\times \Z^{n_2})+\icone(R_i)))\cap\{x:\|b-Ax\|_\infty\leq\kappa_\rho\}\\
=&\cup_i ((P_i\cap(\R^{n_1}\times \Z^{n_2})+\icone(R_i))\cap\{x:\|b-Ax\|_\infty\leq\kappa_\rho\})\\
=&\cup_i\{y=x+\sum \mu_k r_k:x\in P_i\cap(\R^{n_1}\times \Z^{n_2}),r_k\in R_i, \mu_k\in \Z_+,\|b-Ay\|_\infty\leq \kappa_\rho\}\\
=&\cup_i\{y=x+\sum \mu_k r_k:x\in P_i\cap(\R^{n_1}\times \Z^{n_2}),r_k\in R_i, \mu_k\in \Z_+,\|b-Ax\|_\infty\leq \kappa_\rho\}\\
=&\cup_i (P_i\cap(\R^{n_1}\times \Z^{n_2})\cap\{x:\|b-Ax\|_\infty\leq\kappa_\rho\})+\icone(R_i)).
\end{aligned}
\]

Now consider the problem $\inf_x\{c^\top x+\frac{1}{2}x^\top Qx+\bar{\lambda}^\top(b-Ax):x\in P_i\cap(\R^{n_1}\times \Z^{n_2})\cap\{x:\|b-Ax\|_\infty\leq\kappa_\rho\})+\icone(R_i)\}$. If there exists $r\in R_i$ such that $r^\top Q r=0$ (i.e. $Qr=0$), then the feasible region can be rewritten as $P_i\cap(\R^{n_1}\times \Z^{n_2}\cap\{x:\|b-Ax\|_\infty\leq\kappa_\rho\})+\icone(R_i\backslash \{r\})+\{\mu r:\mu\in\Z_+\}$. 

We can use $y+\mu r$ such that $y\in P_i\cap(\R^{n_1}\times \Z^{n_2}\cap\{x:\|b-Ax\|_\infty\leq\kappa_\rho\})+\icone(R_i\backslash \{r\})$ and $\mu\in\Z_+$ to represent $x$. The problem is therefore $\inf_{y,\mu}\{(c^\top-\bar{\lambda}^\top A) \mu+c^\top y+\frac{1}{2}y^\top Qy+\bar{\lambda}^\top(b-Ay):y\in P_i\cap(\R^{n_1}\times \Z^{n_2}\cap\{x:\|b-Ax\|_\infty\leq\kappa_\rho\})+\icone(R_i\backslash \{r\}), \mu\in\Z_+\}$. Optimize the problem over $\mu$ and we get $\mu=0$ an optimal solution (or the problem is unbounded, contrary to \cref{eq:zdg:1b}). Therefore, we can refine the feasible region by omitting all $r\in R_i$ such that $Qr=0$. Denote the set after the process as $R_i^J$. Note that this process is independent of the value of $\rho$, and hence we have 
\[
\begin{aligned}
z^\IP\geq\inf_{x}\ &c^\top x+\frac{1}{2}x^\top Qx+\bar{\lambda}^\top(b-Ax)\ \st\ x\in X,\|b-Ax\|_\infty\leq\kappa_\rho\\
=\inf_{x}\ &c^\top x+\frac{1}{2}x^\top Qx+\bar{\lambda}^\top(b-Ax)\\
\st\ &x\in \cup_i (P_i\cap(\R^{n_1}\times \Z^{n_2})\cap\{x:\|b-Ax\|_\infty\leq\kappa_\rho\})+\icone(R_i^J)).
\end{aligned}
\]

In addition, from (c) of \cref{lm:Qdecomp}, for all $x\in \cone(R_i^J)\backslash\{0\}$, we have $x^\top Q x>0$. Let
\[
V_i=\{x\in (P_i+\cone(R_i^J)):c^\top x+\frac{1}{2}x^\top Qx+\bar{\lambda}^\top(b-Ax)-(z^\IP+1)\leq 0\}.
\]
Note that the definition on $V_i$ is independent of $\rho$. Note that $(P_i\cap(\R^{n_1}\times \Z^{n_2})\cap\{x:\|b-Ax\|_\infty\leq\kappa_\rho\})+\icone(R_i^J))\subseteq (P_i+\cone(R_i^J))$ and we have
\[
\begin{aligned}
z^\IP\geq \inf_{x}\ &c^\top x+\frac{1}{2}x^\top Qx+\bar{\lambda}^\top(b-Ax)\\
\st\ &x\in \cup_i (P_i\cap(\R^{n_1}\times \Z^{n_2})\cap\{x:\|b-Ax\|_\infty\leq\kappa_\rho\})+\icone(R_i^J))\\
=\inf_{x}\ &c^\top x+\frac{1}{2}x^\top Qx+\bar{\lambda}^\top(b-Ax)\\
\st\ &x\in \cup_i (P_i\cap(\R^{n_1}\times \Z^{n_2})\cap\{x:\|b-Ax\|_\infty\leq\kappa_\rho\})+\icone(R_i^J))\\
&c^\top x+\frac{1}{2}x^\top Qx+\bar{\lambda}^\top(b-Ax)\leq z^\IP+1\\
\geq \inf_{x}\ &c^\top x+\frac{1}{2}x^\top Qx+\bar{\lambda}^\top(b-Ax)\\
\st\ &x\in \cup_i (P_i\cap(\R^{n_1}\times \Z^{n_2})\cap\{x:\|b-Ax\|_\infty\leq\kappa_\rho\})+\icone(R_i^J))\\
&x\in \cup_i V_i.\\
\end{aligned}
\]

Using \cref{lm:boundedref} we have that there exists $M_i>0$ such that $V_i\in\{x:\|x\|_\infty\leq M_i\}$. Take $M=\max \{M_i\}$, which is independent of $\rho$, and we have
\[
\begin{aligned}
\inf_{x}\ &c^\top x+\frac{1}{2}x^\top Qx+\bar{\lambda}^\top(b-Ax)\\
\st\ &x\in \cup_i (P_i\cap(\R^{n_1}\times \Z^{n_2})\cap\{x:\|b-Ax\|_\infty\leq\kappa_\rho\})+\icone(R_i^J))\\
&x\in \cup_i V_i\\
\geq \inf_{x}\ &c^\top x+\frac{1}{2}x^\top Qx+\bar{\lambda}^\top(b-Ax)\\
\st\ &x\in \cup_i (P_i\cap(\R^{n_1}\times \Z^{n_2})\cap\{x:\|b-Ax\|_\infty\leq\kappa_\rho\})+\icone(R_i^J))\\
&\|x\|_\infty\leq M\\
\geq\inf_{x}\ &c^\top x+\frac{1}{2}x^\top Qx+\bar{\lambda}^\top(b-Ax)\ \st\ x\in X, \|b-Ax\|_\infty\leq \kappa_\rho,\|x\|_\infty\leq M.
\end{aligned}
\]

While $\inf_{x}\{c^\top x+\frac{1}{2}x^\top Qx+\bar{\lambda}^\top(b-Ax)\ \st\ x\in X, \|b-Ax\|_\infty\leq \kappa_\rho,\|x\|_\infty\leq M\}\geq z^\NLP$ is bounded and the values of $x_2$ here is finite, we can therefore replace $\inf$ by $\min$.

Therefore,
\[
\begin{aligned}
&\inf_{x}\{c^\top x+\frac{1}{2}x^\top Qx+\bar{\lambda}^\top(b-Ax):x\in X,\|b-Ax\|_\infty\leq\kappa_\rho\}\\
\geq&\min_{x}\{c^\top x+\frac{1}{2}x^\top Qx+\bar{\lambda}^\top(b-Ax):x\in X,\|b-Ax\|_\infty\leq\kappa_\rho,\|x\|_\infty\leq M\}.
\end{aligned}
\]

Since it is obvious that 
\[
\begin{aligned}
&\inf_{x}\{c^\top x+\frac{1}{2}x^\top Qx+\bar{\lambda}^\top(b-Ax):x\in X,\|b-Ax\|_\infty\leq\kappa_\rho\}\\
\leq&\min_{x}\{c^\top x+\frac{1}{2}x^\top Qx+\bar{\lambda}^\top(b-Ax):x\in X,\|b-Ax\|_\infty\leq\kappa_\rho,\|x\|_\infty\leq M\},
\end{aligned}
\]
thus equality holds and the proof is completed.
\end{proof}
\section{Exact Penalty Representation}
\label{sec:exactpenalty}
In this section, we will discuss conditions for an exact penalty representation. To begin with, a sufficient condition is given. We later prove the sufficiency of using norm as the penalty function for an exact penalty, while noting that a norm function always satisfies \cref{assum:penalty}.

\begin{theorem}[A Sufficient Condition for Exact Penalty]\label{thm:fixeald_ge}
Under \cref{ass:problem} (MIQP is feasible and the optimal value is bounded), if there exists $\delta$, such that
\[\inf\{\psi(b-Ax):x\in X, Ax\neq b\}\geq\delta>0\]
and $\psi(0)=0$, then there exists a finite $\rho^*$ such that $z^{\LRp}_{\rho^*}(\bar{\lambda})=z^\IP$, which also gives $z^\LDp_{\rho^*}=z^\IP$.
\end{theorem}
\begin{proof}
Under \cref{ass:problem}, using \cref{lm:NLPQPIP}, we have $z^\NLP$ is bounded. Thus, choose a feasible point $\tilde{x}$ for the MIQP and set $$\rho^*=\frac{1}{\delta}(c^\top \tilde{x}+\frac{1}{2}\tilde{x}^\top Q \tilde{x}-z^\NLP)<\infty,$$
we next show that $\rho^*$ satisfies our requirements.

First of all, as $z^\NLP$ bounded and $c^\top \tilde{x}+\frac{1}{2}\tilde{x}^\top Q \tilde{x}\geq z^\IP$, we have $\rho^*\in[0,+\infty)$. Clearly $z^{\LRp}_{\rho^*}(\bar{\lambda})\leq z^\IP$. We next show that $z^{\LRp}_{\rho^*}(\bar{\lambda})\geq z^\IP$.

For $x\in X$ with $Ax=b$, we have $$c^\top x+\frac{1}{2}x^\top Qx+\bar{\lambda}^\top(b-Ax)+\rho^*\psi(b-Ax)=c^\top x+\frac{1}{2}x^\top Qx\geq z^\IP.$$

For $x\in X$ with $Ax\neq b$, we have $$c^\top x+\frac{1}{2}x^\top Qx+\bar{\lambda}^\top(b-Ax)\geq z^\NLP$$
from the strong duality results for QP. Thus,
$$c^\top x+\frac{1}{2}x^\top Qx+\bar{\lambda}^\top(b-Ax)+\rho^*\psi(b-Ax)\geq z^\NLP+\rho^*\delta = c^\top \tilde{x}+\frac{1}{2}\tilde{x}^\top Q \tilde{x}\geq z^\IP.$$

Therefore, we have for $x\in X$, $c^\top x+\frac{1}{2}x^\top Qx+\bar{\lambda}^\top(b-Ax)+\rho^*\psi(b-Ax)\geq z^\IP$ and thus $z^{\LRp}_{\rho^*}(\bar{\lambda})=z^\IP$.
\end{proof}

We now present the exact penalty results for $\psi(\cdot)=\|\cdot\|_\infty$.
\begin{theorem}[Exact Penalty Representation for $L^\infty$-Norm]\label{thm:fixeald}
Assuming $\psi(\cdot)=\|\cdot\|_\infty$, there exists a finite $\rho^*(\bar{\lambda})$ of small complexity, such that $z_{\rho^*}^{\LRp}(\bar{\lambda})=z^{\IP}$.
\end{theorem}
\begin{proof}
It is sufficient to find a finite $\rho^*(\bar{\lambda})$ polynomially bounded, such that $z_{\rho^*}^{\LRp}(\bar{\lambda})\geq z^{\IP}$. Since $z^{\LRp}_\rho(\bar{\lambda})$ is non-decreasing with $\rho$ increasing, without loss of generality, we only consider $\rho\geq 1$. In addition, from \cite[Theorem 4]{del2017mixed}, $z^\IP$, $z^\NLP$ and $\bar{\lambda}$ have $\mathcal{F}$-small complexity.

The constraints $\|b-Ax\|_\infty\leq w$ can be written as $-\mathbf{1}w\leq b-Ax\leq \mathbf{1}w$. Therefore,
\begin{equation}\label{eq:LRp:origin}
\begin{aligned}
z_\rho^{\LRp}(\bar{\lambda})=\inf_{x,w}\ &(c^\top -\bar{\lambda}^\top A)x+\frac{1}{2}x^\top Qx+\rho w+\bar{\lambda}^\top b\\
\st\ &Ax-\mathbf{1} w\leq b,\\
&-Ax-\mathbf{1} w\leq b,\\
&Ex\leq f,\\
&x\in\R^{n_1}\times \Z^{n_2}.
\end{aligned}
\end{equation}

The following lemma shows a uniform bound $M$ of small complexity can be put on $x$ independent of $\rho$.
\begin{lemma}[A Uniform Bound on $x$ Independent of $\rho$]\label{lm:eald:boundary}
Under the assumption that $\rho\geq 1$, there exists $M>0$ independent of $\rho$ and of small complexity, such that
\begin{equation}\label{eq:LRp:bounded}
\begin{aligned}
z_\rho^{\LRp}(\bar{\lambda})=\inf_{x,w}\ &(c^\top -\bar{\lambda}^\top A)x+\frac{1}{2}x^\top Qx+\rho w+\bar{\lambda}^\top b\\
\st\ &Ax-\mathbf{1} w\leq b,\\
&-Ax-\mathbf{1} w\leq b,\\
&Ex\leq f,\\
&\|x\|_\infty\leq M,\\
&x\in\R^{n_1}\times \Z^{n_2}.
\end{aligned}
\end{equation}
\end{lemma}
A proof of \cref{lm:eald:boundary} is provided later. We next rewrite $x=(x_1,x_2)$ and separate $A,E,c$ respectively. We also rewrite 
\[Q=\begin{bmatrix}Q^{(11)}&Q^{(12)}\\Q^{(21)}&Q^{(22)}\end{bmatrix}.\]
Note that $Q^{(11)}$ is also positive semi-definite. Therefore, the problem can be rewrite as
\begin{equation}\label{eq:LRp:separatebounded}
\begin{aligned}
z_\rho^{\LRp}(\bar{\lambda})=\inf_{x_1,x_2,w}\ &(c_1^\top -\bar{\lambda}^\top A_1+x_2^\top Q^{(21)})x_1+\frac{1}{2}x_1^\top Q^{(11)}x_1\\
&+\rho w+\bar{\lambda}^\top b+(c_2^\top-\bar{\lambda}^\top A_2)x_2+\frac{1}{2}x_2^\top Q^{(22)}x_2\\
\st\ &A_1x_1-\mathbf{1} w\leq -A_2x_2 + b,\\
&-A_1x_1-\mathbf{1} w\leq A_2x_2 -b,\\
&E_1x_1\leq f-E_2x_2,\\
&x_1\leq\mathbf{1}M,-x_1\leq \mathbf{1}M,\\
&\|x_2\|_\infty\leq M,\\
&x_1\in\R^{n_1},x_2\in\Z^{n_2}\\
\end{aligned}
\end{equation}
Denote $V=\{v\in\Z^{n_2}: \|v\|_\infty\leq M\}$. In addition, we use $z_\rho^{\LRp}(\bar{\lambda},x_2)$ to denote $z_\rho^{\LRp}(\bar{\lambda})$ while fixing $x_2$. Therefore, $z_\rho^{\LRp}(\bar{\lambda})=\min_{x_2\in V}z_\rho^{\LRp}(\bar{\lambda},x_2)$. Note that $z_\rho^{\LRp}(\bar{\lambda},x_2)$ is still non-decreasing with respect to $\rho$. Therefore, from \cref{thm:zdg} we have $z^\IP=\lim_{\rho\to+\infty}\min_{x_2\in V}z_\rho^{\LRp}(\bar{\lambda},x_2)=\min_{x_2\in V}\lim_{\rho\to+\infty}z_\rho^{\LRp}(\bar{\lambda},x_2)$, or equivalently we have that $\lim_{\rho\to+\infty}z_\rho^{\LRp}(\bar{\lambda},x_2)\geq z^\IP$.

For arbitrary $x_2\in V$, the dual problem of \cref{eq:LRp:separatebounded} (with respect to $x_1,w$) is therefore
\[
\begin{aligned}
z^{\text{DRD+}}_\rho(\bar{\lambda},x_2):&=\sup_{y_1,y_2,y_3,y_4,y_5\geq 0}\inf_{x_1,w}(c_1^\top -\bar{\lambda}^\top A_1+x_2^\top Q^{(21)})x_1+\frac{1}{2}x_1^\top Q^{(11)}x_1\\
&+\rho w+\bar{\lambda}^\top b+(c_2^\top-\bar{\lambda}^\top A_2)x_2+\frac{1}{2}x_2^\top Q^{(22)}x_2\\
&+y_1^\top(A_1x_1-\mathbf{1} w+A_2x_2-b)-y_2^\top(A_1x_1+\mathbf{1} w+A_2x_2-b)\\
&+y_3^\top(E_1x_1-f+E_2x_2)+y_4^\top(x_1-\mathbf{1}M)+y_5^\top(-x_1-\mathbf{1}M).
\end{aligned}
\]
Note that the problem $\inf_{x_1,w}$ is bounded if and only if $(y_1+y_2)^\top \mathbf{1}=\rho$ and there exists $ \nu$, s.t. $c_1^\top -\bar{\lambda}^\top A_1+x_2^\top Q^{(21)}+(y_1-y_2)^\top A_1+y_3^\top E_1+(y_4-y_5)^\top=\nu^\top Q^{(11)}$. Therefore the problem is
\[
\begin{aligned}
z^{\text{DRD+}}_\rho(\bar{\lambda},x_2)=\sup_{y,\nu}\ &- \frac{1}{2}\nu^\top Q^{(11)}\nu+(A_2x_2-b)^\top y_1-(A_2x_2-b)^\top y_2\\
&+(E_2x_2-f)^\top y_3-M\mathbf{1}^\top(y_4+y_5)\\
&+\bar{\lambda}^\top b+(c_2^\top-\bar{\lambda}^\top A_2)x_2+\frac{1}{2}x_2^\top Q^{(22)}x_2\\
\st\ &y_1,y_2,y_3,y_4,y_5\geq 0,\\
& \mathbf{1}^\top (y_1+y_2)=\rho,\\
& c- A_1^\top \bar{\lambda}+Q^{(12)} x_2+A_1^\top (y_1-y_2)+E_1^\top y_3 +y_4-y_5=Q^{(11)\top} \nu.
\end{aligned}
\]

By strong duality we have $z_\rho^{\LRp}(\bar{\lambda},x_2)=z^{\text{DRD+}}_\rho(\bar{\lambda},x_2)$. Therefore, we have $\lim_{\rho\to+\infty}z_\rho^{\text{DRD+}}(\bar{\lambda},x_2)=\lim_{\rho\to+\infty}z_\rho^{\LRp}(\bar{\lambda},x_2)\geq z^\IP$. Now consider the following problem with respect to $(\xi,y,\nu,\rho)$:
\[
\begin{aligned}
\min\ &\xi\\
\st\ &- \frac{1}{2}\nu^\top Q^{(11)}\nu+(A_2x_2-b)^\top y_1-(A_2x_2-b)^\top y_2\\
&+(E_2x_2-f)^\top y_3-M\mathbf{1}^\top(y_4+y_5)\\
&+\bar{\lambda}^\top b+(c_2^\top-\bar{\lambda}^\top A_2)x_2+\frac{1}{2}x_2^\top Q^{(22)}x_2+\xi\geq z^\IP,\\
&y_1,y_2,y_3,y_4,y_5,\xi\geq 0,\ \rho\geq 1,\\
& \mathbf{1}^\top (y_1+y_2)=\rho,\\
& c- A_1^\top \bar{\lambda}+Q^{(12)} x_2+A_1^\top (y_1-y_2)+E_1^\top y_3 +y_4-y_5=Q^{(11)\top} \nu.
\end{aligned}
\]
The above problem is a quadratically constrained quadratic programming (QCQP) with convex constraints and affine objective function. The existence of an optimal solution is guaranteed by the finiteness and feasibility of the problem \cite{bertsekas2003convex}. In addition, as $\lim_{\rho\to+\infty}z_\rho^{\text{DRD+}}(\bar{\lambda},x_2)\geq z^\IP$ there exists a sequence of $(\xi^k,y^k,\nu^k,\rho^k)$ feasible to the problem such that $\xi^k\to 0$. Therefore, the optimal value is $0$ and an optimal solution $(0,y^*,\nu^*,\rho^*)$ exists, which guarantees the feasibility of the following problem with respect to $(y,\nu,\rho)$:
\[
\begin{aligned}
\rho^*(x_2):=\min\ &\rho\\
\st\ &- \frac{1}{2}\nu^\top Q^{(11)}\nu+(A_2x_2-b)^\top y_1-(A_2x_2-b)^\top y_2\\
&+(E_2x_2-f)^\top y_3-M\mathbf{1}^\top(y_4+y_5)\\
&+\bar{\lambda}^\top b+(c_2^\top-\bar{\lambda}^\top A_2)x_2+\frac{1}{2}x_2^\top Q^{(22)}x_2\geq z^\IP,\\
&y_1,y_2,y_3,y_4,y_5\geq 0,\ \rho\geq 1,\\
& \mathbf{1}^\top (y_1+y_2)=\rho,\\
& c- A_1^\top \bar{\lambda}+Q^{(12)} x_2+A_1^\top (y_1-y_2)+E_1^\top y_3 +y_4-y_5=Q^{(11)\top} \nu.
\end{aligned}
\]
Similarly, the finiteness and feasibility of the problem guarantees the existence of the optimal solution. Therefore, $\rho^*(x_2)$ is well defined and from \cite[Theorem 4]{del2017mixed}, $\rho^*(x_2)$ has small complexity. In addition, we have $z^{\text{DRD+}}_{\rho^*(x_2)}(\bar{\lambda},x_2)\geq z^\IP$.

Now let $\rho^*=\max_{x_2\in V}\rho^*(x_2)$ of small complexity, we have $z^{\text{DRD+}}_{\rho^*} (\bar{\lambda},x_2)\geq z^{\text{DRD+}}_{\rho^*(x_2)}(\bar{\lambda},x_2)\geq z^\IP$ for all $x_2\in V$. Hence, $z^{\LRp}_{\rho^*} (\bar{\lambda})=\min_{x_2\in V}z^{\LRp}_{\rho^*} (\bar{\lambda},x_2)=\min_{x_2\in V}z^{\text{DRD+}}_{\rho^*} (\bar{\lambda},x_2)\geq z^\IP$.

Note that for any $\rho,\lambda$, $z_\rho^{\LRp}(\lambda)\leq z_\rho^{\LDp}\leq z^{\IP}$, and thus $z_{\rho^*}^\LRp(\bar{\lambda})=z_{\rho^*}^{\LDp}=z^\IP$.
\end{proof}

We next complete the proof by proving \cref{lm:eald:boundary}.
\begin{proof}[Proof of \cref{lm:eald:boundary}]
Consider for $\rho\geq 1$, $w_M:=2(z^\IP-z^\NLP) \geq 2w^*_\rho$, where the inequality follows \cref{eq:wrhostar}, and define
\[
\begin{aligned}
\hat{z}_\rho^{\LRp}(\bar{\lambda})=\inf_{x,w}\ &(c^\top -\bar{\lambda}^\top A)x+\frac{1}{2}x^\top Qx+\rho w+\bar{\lambda}^\top b\\
\st\ &Ax-\mathbf{1} w\leq b,\\
&-Ax-\mathbf{1} w\leq -b,\\
&Ex\leq f,\\
&w\leq w_M,\\
&x\in\R^{n_1}\times \Z^{n_2}.
\end{aligned}\]
Denote the feasible region for this problem as $P$ and the feasible region for the original problem \cref{eq:LRp:origin} as $P^o$. Clearly, $P\subseteq P^o$ and thus $\hat{z}_\rho^{\LRp}(\bar{\lambda})\geq z_\rho^{\LRp}(\bar{\lambda})$. Similarly as $P$ is larger than the feasible region of \cref{eq:penaltyrelax:tildez}, we have $\hat{z}_\rho^{\LRp}(\bar{\lambda})\leq \tilde{z}_\rho^{\LRp}(\bar{\lambda})=z_\rho^{\LRp}(\bar{\lambda})$ and thus $\hat{z}_\rho^{\LRp}(\bar{\lambda}) = z_\rho^{\LRp}(\bar{\lambda})$.

Now apply \cref{lm:Qdecomp} to $P$ and we get a decomposition $P=\cup_i (P_i\cap(\R^{n_1}\times \Z^{n_2}\times \R)+\icone(R_i))$ with the properties listed in the lemma. Note that the decomposition applies to all $\rho$. 

Note that the problem is bounded from the boundedness of $z_\rho^{\LRp}$ and for all $r\in R_i$ the $w$-component is $0$ (from the constraints $w\leq w_M$, so we can omit the $w$-component for any vector in $R_i$, or simply denote it as $R_i\times \{0\}$.

Similar to the proof to \cref{lm:zdg:boundary}, when we solve the problem $\inf_{x,w}\{(c^\top -\bar{\lambda}^\top A)x+\frac{1}{2}x^\top Qx+\rho w+\bar{\lambda}^\top b:(x,w)\in P_i\cap(\R^{n_1}\times \Z^{n_2}\times \R)+\icone(R_i)\times \{0\}\}$, if there exists $r\in R_i$ such that $r^\top Q r=0$ (i.e. $Qr=0$), the feasible region can be decomposed as as $P_i\cap(\R^{n_1}\times \Z^{n_2}\times \R)+\icone(R_i\backslash \{r\})\times\{0\}+\{\mu r:\mu\in\Z_+\}\times \{0\}$. Optimize the problem over $\mu$ and we get $\mu=0$ an optimal solution (otherwise the problem will be unbounded). Therefore, we can refine the feasible region by omitting all $r\in R_i$ such that $Qr=0$. Denote the set after the process as $R_i^J$. Note that this process is independent of $\rho$, and hence we have 
\[
\begin{aligned}
z_\rho^{\LRp}(\bar{\lambda})=\inf_{x,w}\ &(c^\top -\bar{\lambda}^\top A)x+\frac{1}{2}x^\top Qx+\rho w+\bar{\lambda}^\top b\\
\st\ &(x,w)\in \cup_i(P_i\cap(\R^{n_1}\times \Z^{n_2}\times \R)+\icone(R_i^J)\times\{0\}).
\end{aligned}\]
Now, from (c) of \cref{lm:Qdecomp}, for all $x\in \cone(R_i^J)\backslash\{0\}$, we have $x^\top Q x>0$.
Let
\[
\begin{aligned}
V_i=\{&(x,w)\in (P_i+\cone(R^J_i)\times\{0\}):\\
&(c^\top -\bar{\lambda}^\top A)x+\frac{1}{2}x^\top Qx+w+\bar{\lambda}^\top b\leq z^\IP+1\}.
\end{aligned}
\]
Note that the definition on $V_i$ is independent of $\rho$. Therefore, as $\rho\geq 1$ we have
\[\begin{aligned}
z_\rho^{\LRp}(\bar{\lambda})=\inf_{x,w}\ &(c^\top -\bar{\lambda}^\top A)x+\frac{1}{2}x^\top Qx+\rho w+\bar{\lambda}^\top b\\
\st\ &(x,w)\in \cup_i(P_i\cap(\R^{n_1}\times \Z^{n_2}\times \R)+\icone(R_i^J)\times\{0\})\\
=\inf_{x,w}\ &(c^\top -\bar{\lambda}^\top A)x+\frac{1}{2}x^\top Qx+\rho w+\bar{\lambda}^\top b\\
\st\ &(x,w)\in \cup_i(P_i\cap(\R^{n_1}\times \Z^{n_2}\times \R)+\icone(R_i^J)\times\{0\})\\
&(c^\top -\bar{\lambda}^\top A)x+\frac{1}{2}x^\top Qx+w+\bar{\lambda}^\top b\leq z^\IP+1\\
\geq \inf_{x,w}\ &(c^\top -\bar{\lambda}^\top A)x+\frac{1}{2}x^\top Qx+\rho w+\bar{\lambda}^\top b\\
\st\ &(x,w)\in \cup_i(P_i\cap(\R^{n_1}\times \Z^{n_2}\times \R)+\icone(R_i^J)\times\{0\})\\
&(x,w)\in\cup_i V_i.
\end{aligned}\]

From \cref{lm:boundedref}, there exists $ M_i>0$, such that $V_i\subseteq \{(x,w):\|(x,w)\|_\infty\leq M_i\}$ and  $M_i$ has small complexity. Therefore, $V_i \subset \{(x,w):\|x\|_\infty\leq M_i\}$. Let $M=\max\{M_i\}$ (which is again independent of $\rho$ and has small complexity) and we have
\[
\begin{aligned}
z_\rho^{\LRp}(\bar{\lambda})\geq \inf_{x,w}\ &(c^\top -\bar{\lambda}^\top A)x+\frac{1}{2}x^\top Qx+\rho w+\bar{\lambda}^\top b\\
\st\ &(x,w)\in \cup_i(P_i\cap(\R^{n_1}\times \Z^{n_2}\times \R)+\icone(R_i^J)\times\{0\})\\
&(x,w)\in\cup_i V_i\\
\geq \inf_{x,w}\ &(c^\top -\bar{\lambda}^\top A)x+\frac{1}{2}x^\top Qx+\rho w+\bar{\lambda}^\top b\\
\st\ &(x,w)\in \cup_i(P_i\cap(\R^{n_1}\times \Z^{n_2}\times \R)+\icone(R_i^J)\times\{0\})\\
&\|x\|_\infty\leq M\\
\geq \inf_{x,w}\ &(c^\top -\bar{\lambda}^\top A)x+\frac{1}{2}x^\top Qx+\rho w+\bar{\lambda}^\top b\\
\st\ &(x,w)\in P^o,\ \|x\|_\infty\leq M.
\end{aligned}\]

Since
\[
\begin{aligned}
z_\rho^{\LRp}(\bar{\lambda})=\inf_{x,w}\ &(c^\top -\bar{\lambda}^\top A)x+\frac{1}{2}x^\top Qx+\rho w+\bar{\lambda}^\top b\\
\st\ &(x,w)\in P^o\\
\leq \inf_{x,w}\ &(c^\top -\bar{\lambda}^\top A)x+\frac{1}{2}x^\top Qx+\rho w+\bar{\lambda}^\top b\\
\st\ &(x,w)\in P^o,\ \|x\|_\infty\leq M,
\end{aligned}
\]
equality holds and the proof is completed.
\end{proof}

Next, we will generalize the result to any norm penalty and any dual variable.
\EPR*
\begin{proof}
Denote the $\rho^*(\bar{\lambda})$ in \cref{thm:fixeald} as $\hat{\rho}$ to avoid confusion.

As $\psi(\cdot)$ is a norm function, there exists $ \gamma\in[1,+\infty)$ such that $\gamma\|\cdot\|_\infty\geq \psi(\cdot)\geq\|\cdot\|_\infty/\gamma$. Without loss of generality, we round up $\gamma$ to a closest integer, which is still a constant decided only by $\|\cdot\|_\infty$ and $\psi(\cdot)$. Therefore, by letting $\rho^*(\bar{\lambda})=\gamma\hat{\rho}$, which still has small complexity, we have 
\[
\begin{aligned}
z^\LRp_{\rho^*(\bar{\lambda})}(\bar{\lambda})=&\inf_{x\in X}\{c^\top x+\frac{1}{2}x^\top Q x+\bar{\lambda}^\top(b-Ax)+\rho^*(\bar{\lambda})\psi(b-Ax)\}\\
\geq&\inf_{x\in X}\{c^\top x+\frac{1}{2}x^\top Q x+\bar{\lambda}^\top(b-Ax)+\hat{\rho}\|b-Ax\|_\infty\}\\
=&z^\IP,
\end{aligned}
\]
where the last equation comes from \cref{thm:fixeald}. Along with $z_\rho^{\LRp}(\lambda)\leq z_\rho^{\LDp}\leq z^\IP$, we have $z^\LRp_{\rho^*(\bar{\lambda})}(\bar{\lambda})=z^\LDp_{\rho^*(\bar{\lambda})}=z^\IP$ and (a) is proven. Now it only remains to show that we can replace $\bar{\lambda}$ by any dual vector $\tilde{\lambda}\in\R^m$.

From Cauchy-Schwarz inequality, we have
$$-\|\lambda\|_2\|b-Ax\|_2\leq\lambda^\top(b-Ax)\leq\|\lambda\|_2\|b-Ax\|_2.$$

Again, applying the property of the norm, there exists $ \eta\in[1,+\infty)\cap\Z_+$ decided only by $\|\cdot\|_2$ and $\psi(\cdot)$, such that $\eta\|\cdot\|_2\geq \psi(\cdot)\geq\|\cdot\|_2/\eta$, and we have
$$\tilde{\lambda}(b-Ax)-\bar{\lambda}^\top (b-Ax)\geq-\eta\|\tilde{\lambda}-\bar{\lambda}\|_2\psi(b-Ax).$$

By setting $\rho^*(\tilde{\lambda})=\lceil\rho^*(\bar{\lambda})+\eta\|\tilde{\lambda}-\bar{\lambda}\|_2\rceil$, which has $\F,\tilde{\lambda}$-small complexity, we have
\[
\begin{aligned}
z^\LRp_{\rho^*(\tilde{\lambda})}(\tilde{\lambda})=&\inf_{x\in X}\{c^\top x+\frac{1}{2}x^\top Q x+\lambda^\top(b-Ax)+\rho^*(\tilde{\lambda})\psi(b-Ax)\}\\
\geq &\inf_{x\in X}\{c^\top x+\frac{1}{2}x^\top Q x+\lambda^\top(b-Ax)+\rho^*(\bar{\lambda})\psi(b-Ax)\}\\
=&z^\LRp_{\rho^*(\bar{\lambda})}(\bar{\lambda})=z^\IP.
\end{aligned}
\]

Therefore, along with $z_\rho^{\LRp}(\lambda)\leq z^\IP$, we have $z^\LRp_{\rho^*(\tilde{\lambda})}(\tilde{\lambda})=z^\IP$.
\end{proof}
\begin{remark}
The results also apply to MILP, which yield that exact penalty weight $\rho^*$ (which is detailedly discussed in \cite{feizollahi2017exact}) also has $\F$-small complexity.
\end{remark}
\section{Conclusions}
\label{sec:conclusions}
In this paper, we investigate ALD for MIQP. We prove that an asymptotic zero duality gap is reachable as the penalty weight goes to infinity, under some mild conditions (\cref{assum:penalty}) on the penalty function. We also show that a finite penalty weight is enough for an exact penalty when we use any norm as the penalty function. Moreover, we prove that a penalty weight of polynomial size is enough to give an exact penalty representative.

By dualizing and penalizing the difficult constraints using ALD, we can convert the problem to one with an easy feasible region, while maintains the optimality of the original optimal points. However, by introducing a penalty term in the objective function, it might become more difficult to deal with. In addition, as ALD does not deal with integer constraints, the problem is still far from convex.

A special case where the easy constraints are separable, leads us to consider the alternating direction method of multipliers (ADMM) \cite{boyd2011distributed} and relative update schemes, which are proposed to solve convex problems separably. However, for mixed integer problems, such methods are mainly heuristic, like \cite{takapoui2017simple} for MIQP based on ADMM. Future development of separable exact algorithms utilizing the strong duality results and solving general non-convex problems is a potential direction of research.

\clearpage
\bibliographystyle{siamplain}
\bibliography{bib.bib}
\end{document}